\documentclass[reqno, 12pt]{amsart}
\usepackage{amsmath,amsthm,amsfonts,amssymb}

\date{}
\newtheorem{theorem}{Theorem}[section]
\newtheorem{lemma}{Lemma}[section]
\newtheorem{corollary}{Corollary}[section]
\newtheorem{remark}{Remark}[section]

\newtheorem{example}{Example}[section]

\begin{document}

\centerline{\sc Existence theorems for nonlinear stationary Kolmogorov equations}
	
\centerline{\sc with partially degenerate diffusion matrices}

\vskip 2 ex

\begin{center}
{\sc Aziz M. Embarek${}^{a}$\footnote{Corresponding author
\par
e-mails: embarek.aziz@gmail.com (A.M.Embarek), shatiltop@mail.ru (D.V.Shatilovich).}, \,Dmitry V. Shatilovich${}^{a}$
}
\end{center}

\vskip 1. ex

\quad ${}^{a}$ Faculty of Mechanics and Mathematics, Lomonosov Moscow State
University, 119991, GSP-1, Leninskie Gory, Moscow, Russia;

\vskip 4 ex

{\bf Abstract} We study nonlinear stationary Kolmogorov equations with degenerate diffusion matrices and discontinuous coefficients. The existence of a solution is proved. We propose a new approach based on an integral condition with Lyapunov functions and a regularity of projections of solutions in the partially degenerate case. Examples are given to illustrate the results.

\vskip 2 ex

{\bf Keywords:} nonlinear Kolmogorov equations, invariant measures of diffusion processes, Lyapunov functions.

\vskip 2 ex

{\bf AMS Subject Classification:} 35J60, 35J70, 35R05

\vskip 2 ex
	
	\section{\sc Introduction}
	
	\vspace*{0.2cm}
	
We consider the following equation
\begin{equation}\label{eq1}
	\sum_{i,j=1}^d\partial_{x_i}\partial_{x_j}\bigl(a^{ij}(x,\mu)\mu\bigr)-\sum_{i=1}^d\partial_{x_i}\bigl(b^i(x,\mu)\mu\bigr)=0.
\end{equation}
For any probability measure $\mu$ the mappings $x\mapsto a^{ij}(x,\mu)$ and $x\mapsto b^i(x,\mu)$ are either undefined or are Borel functions, and the matrix $A = (a^{ij}(x,\mu))$ is symmetric and nonnegative definite. We say that a nonnegative locally finite Borel measure $\mu$ is a solution to equation (\ref{eq1}) if the mappings $x\mapsto a^{ij}(x,\mu)$ and $x\mapsto b^i(x,\mu)$ are integrable with respect to $\mu$ over any compact set and the identity
	\begin{equation}\label{integr}
	\int_{\mathbb{R}^d}L_\mu u(x)\mu(dx)=0.
\end{equation}
holds for every function $u\in C_0^{\infty}(\mathbb{R}^d)$. Here we denote
$$
L_\mu u(x)=\sum_{i, j=1}^da^{ij}(x,\mu)\partial_{x_i}\partial_{x_j}u(x)+\sum_{i=1}^db^i(x,\mu)\partial_{x_i}u(x).
$$
We will prove that equation (\ref{eq1}) has a probability solution, that is, a solution for which the equality $\mu(\mathbb R^d)=1$ holds, in the case when the matrix $A = (a^{ij})$ is partially degenerate, the coefficients are unbounded and discontinuous in $x$. The integral condition with a Lyapunov function will play a key role in the proof.

Interest in probability solutions is connected with the fact that an invariant measure of a diffusion process with the generator $L_\mu$ satisfies equation (\ref{eq1}). A proof of this fact can be found in \cite[chapter 5]{book}. By the superposition principle (see \cite{ZV23}, \cite{BRSsup23}) in a general situation each probability solution $\mu$ gives rise to a measure $P$ on $C([0, +\infty), \mathbb{R}^d)$ such that $P$ is a solution to the martingale problem with $L_\mu$ and for all $t\ge 0$ the equality $P(\omega\colon \omega(t)\in B)=\mu(B)$ holds. According to \cite[chapter 4]{IW89}, this allows us to construct a probability space with a Wiener process $W$ and an adapted process $X$ such that 
$$
dX_t = \sqrt{2A(X_t,\mu)}dW_t + b(X_t,\mu)dt
$$
and the distribution of the random variable $X_t$ coincides with $\mu$ for all $t \ge 0$. The last equation is often referred as to the McKean--Vlasov or mean-field stochastic differential equation. The existence of a probability solution to equation (\ref{eq1}) also allows, in the case of irregular coefficients growing at infinity, to construct a sub-Markov semigroup with generator $L_\mu$ (see \cite[chapter 5]{book} and \cite{LiSt}).

The work on stochastic differential equations with coefficients that depend on the law of the solution was initiated by the papers of Kac \cite{Kac56} in kinetic theory, as a stochastic counterpart of the Vlasov equation in plasma physics \cite{Vlasov68}, and of McKean \cite{McK66} for nonlinear parabolic partial differential equations. It describes the limiting behaviour of an individual particle evolving within a large system of particles interacting through its empirical measure, as the size of the population grows to infinity. The behaviour of the particle system is ruled by the propagation of chaos phenomenon as originally studied by McKean \cite{McK67}. The results on the existence and uniqueness of a solution for such stochastic differential equations can be found in \cite{BS24}, \cite{MV20}, \cite{HRW21}. McKean--Vlasov equations play an important role in the actively developing mean-field games theory, which has been introduced to solve the problem of existence of an approximate Nash equilibrium for differential games with a large number of players (see \cite{LL07}, \cite{CD18}, \cite{SS25}, \cite{KT19}). The case $A = 0$ corresponds to deterministic mean field games which, for example, are discussed in papers \cite{A20}, \cite{CG15}. The books \cite{J16}, \cite{K19} study equations of a general form, including the nonlinear Kolmogorov equations.

The existence of invariant measures, corresponding to McKean--Vlasov equations, and the convergence of solutions to an invariant measure, as time tends to infinity, have been investigated by many authors (see \cite{W21}, \cite{ButS17}, \cite{ABR19}, \cite{Z21}). In \cite{AV26} the stability of solutions to equation~(\ref{eq1}) in the case $A = 0$ is considered. In this paper we focus on the existence of probability solutions to equation (\ref{eq1}), rather than on the existence of invariant measures.

A standard global condition, ensuring the existence of solutions to McKean–Vlasov equations and parabolic Kolmogorov equations, is the existence of a Lyapunov function, that is, a nonnegative function
$V\in C^2(\mathbb{R}^d)$ such that
\begin{equation}
	\label{cond:infty}
	\lim_{|x|\to+\infty}V(x)=+\infty
\end{equation}
and for some positive constants $C_1, C_2$ the inequality 
\begin{equation}
	\label{lyap_cond1}
	L_\mu V(x)\le C_1+C_2 V(x)
\end{equation}
holds for all $x\in \mathbb R^d$. In \cite{HSS21} the existence of solutions to McKean--Vlasov equations for a wider class of Lyapunov functions was established. So inequality (\ref{lyap_cond1}) is replaced by the following: there exist positive constants $C_1, C_2$, a nonnegative function $V \in C^2(\mathbb R^d)$, satisfying~(\ref{cond:infty}), and a nonnegative function $W$ such that $W(\cdot,t,\mu)\in C^2(\mathbb R^d)$, $W(x,\cdot,\mu) \in C^1([0,\infty))$, $W(x,t,\cdot)$ is differentiable in a suitable sense (see \cite[Definition A.5.]{HSS21}) and the inequalities
\begin{equation}
	\label{lyap_cond2}
	\int_{\mathbb R^d} L_\mu W d\mu \le C_1 + C_2 \int_{\mathbb R^d}W d\mu\quad\text{ and }\quad \int_{\mathbb R^d} V d\mu \le \int_{\mathbb R^d} W d\mu
\end{equation}
holds for all probability measures $\mu$ with compact support. Here the operator $L_\mu$ includes derivatives with respect to measures. In \cite[Theorem 6.6]{BS24} the existence of a probability solution to parabolic Kolmogorov equations was established in the case when the inequalities (\ref{lyap_cond2}) are satisfied for all sub-probability measures $\mu$ with compact support, $W$ depends only on $x$ and $W = V.$ \cite[Example 6.7]{BS24} shows that inequality (\ref{lyap_cond2}) is less restrictive than (\ref{lyap_cond1}) even in this case and allows us to construct a solution for all $t\ge 0$, while inequality (\ref{lyap_cond1}) gives us a solution only for $t \in [0, \tau]$, where $\tau$ is sufficiently small.

In the case of stationary Kolmogorov equations the situation is more complicated, since in order to prove the existence of solutions it is necessary to assume that the more restrictive inequality 
\begin{equation}
	\label{lyap_cond3}
L_\mu V(x)\le C-\Lambda V(x)
\end{equation}
holds for all $x\in \mathbb R^d$, where $C, \Lambda$ are positive constants. According to \cite[Proposition 5.3.9]{book}, under broad assumptions on the coefficients, the existence of a probability solution to equation (\ref{eq1}) leads to the existence of a Lyapunov function. However, without additional local conditions on the coefficients, the existence of a Lyapunov function does not yet guarantee the existence of a probability solution. So the first example of \cite{ShSh24} shows that, in the case of a degenerate matrix $A$ and a discontinuous coefficient $b$, equation (\ref{eq1}) may have no solutions. Theorem~\ref{main_result} is the main result of this paper, which establishes the existence of a probability solution to a stationary Kolmogorov equation for which there exist a nonnegative function $V \in C^2(\mathbb R^d)$, satisfying~(\ref{cond:infty}), and a nonnegative function $W$ such that $x \mapsto W(x,\mu)$ belongs to $C^2(\mathbb R^d)$ and the inequality
	$$
	\int_{\mathbb R^d}L_\mu W d\mu \le C - \Lambda \int_{\mathbb R^d}V d\mu,
	$$
holds for all probability measures $\mu$ with compact support. This inequality is similar to $(\ref{lyap_cond2})$, but the negative coefficient $(-\Lambda)$ is taken into account. Here the differentiability of $W$ with respect to measures is not assumed and the operator $L_\mu$ contains only derivatives with respect to $x$. Examples \ref{ex2}, \ref{ex3} and \ref{ex4} show that this condition allows us to obtain the existence of a probability solution for a larger class of Kolmogorov equations than condition (\ref{lyap_cond3}).

Another innovation of this work is the consideration of nonlinear stationary Kolmogorov equations with partially degenerate diffusion matrices $A$, that is, the matrix $A$ is degenerate only in some of the variables. We assume that the upper left $m\times m$ corner of the matrix $A(x,\mu)$ is non-degenerate for all probability measures $\mu$ and $x \in \mathbb R^d$. The result~\cite[Thorem~2.5.8]{book} gives sufficient conditions for the existence of a probability solution in the case when the matrix $A$ is non-degenerate, while \cite[Theorem 5.1]{BS24} gives that when the non-degeneracy is not assumed. In \cite{ShSh24} it has been investigated the case of linear stationary Kolmogorov equations with partially degenerate diffusion matrices. In this case the projection of a solution onto the first $m$ variables has a regular density, and local regularity conditions of the coefficients in the first $m$ variables, for the existence of a solution, can be relaxed: the coefficients of the equation can be discontinuous in $x_1,\ldots, x_m$. In this paper it is shown that the same is true for nonlinear stationary Kolmogorov equations. A typical example of a nonlinear partially degenerate Kolmogorov equation is provided by the operator
$$
L_\mu u(x)=\sum_{i=1}^m\partial_{x_i}^2u(x)+\sum_{i=1}^db^i(x, \mu)\partial_{x_i}u(x), \quad 1\le m\le d,
$$
which appears when we study stochastic differential equations, for which 
the diffusion process is described by ordinary differential equations with respect to some variables. For example, we can consider the system of stochastic Langevin equations
$$
d Z_t=Y_t\,dt, \quad dY_t=B(Y_t, Z_t, \mu)\,dt+\sqrt{2}\,dW_t,
$$
where $W_t$ is a Wiener process. The corresponding operator $L$ is
$$
Lu(y, z)=\partial_y^2u(y, z)+y\partial_zu(y, z)+B(y, z, \mu)\partial_yu(y, z).
$$

Establishing the uniqueness of even a probability solution to a stationary Kolmogorov equation is a more difficult problem than proving the uniqueness of an invariant measure for the corresponding McKean-Vlasov equation, since the class of measures under consideration is larger. One recent result, which does not assume non-degeneracy of diffusion matrices, is presented in \cite{BSS25}. A detailed review of known results for degenerate Kolmogorov equations can be found in \cite{R23}.

This paper consists of five sections. In Section 2 we discuss the main results. In Section~3~examples~illustrating the results are given. Auxiliary results are proved in Section 4. Section 5 is devoted to the proofs of the main results.

\section{Main results}
\label{sec2}

Let $V \in C^2(\mathbb R^d),\,V\ge 0$ and $\lim_{|x|\rightarrow\infty}V(x) = +\infty$. For example, one can take the function $V(x) = (1+|x|^2)^{p/2}$, where $p > 0$. The bounded Borel measure $\mu$ on $\mathbb R^d$ is called a probability measure if $\mu \ge 0$ and $\mu(\mathbb R^d) = 1$. For $R > 0$, let us denote by $\mathcal P_R(V)$ the set of probability measures on $\mathbb R^d$ such that
$$
\int_{\mathbb R^d}V d\mu \le R.
$$
The set $\mathcal P_R(V)$ is non-empty for sufficiently large $R$. This set is convex and compact in the space of probability measures with the weak topology. Let $\mathcal P(V)$ denote the set of all probability measures $\mu$ such that $V \in L^1(\mu)$.

We say that a sequence $\mu_n \in \mathcal P(V)$ converges $V$-weakly to a measure $\mu \in \mathcal P(V)$ if
$$
\lim_{n \rightarrow\infty}\int_{\mathbb R^d}f(x)\mu_n(dx) = \int_{\mathbb R^d}f(x)\mu(dx)
$$
for every continuous function $f$ such that 
$$\lim_{|x|\rightarrow\infty}\frac{f(x)}{V(x)} = 0.$$ In \cite[Section 5]{BS24} the following proposition is proved: if for some $R > 0$ a sequence $\mu_n\in \mathcal P_{R}(V)$ converges weakly to a measure $\mu\in \mathcal P_{R}(V)$, then $\mu_n$ converges $V$-weakly to $\mu$.

In what follows, we assume that the upper left $m\times m$ corner of the matrix $A(x,\mu)$ is non-degenerate for all probability measures $\mu$ and $x \in \mathbb R^d$. Let us split the coordinates of variables into two parts $x = (y, z)$, where $y = (x_1, \ldots, x_m)$ and $z = (x_{m+1}, \ldots, x_d)$. 

Let us formulate our main assumptions.

\vspace*{0.2cm}

{\bf (H1) Local conditions:}

\vspace*{0.2cm}

{\bf (H1.1)} There exists an integer $0 \le m \le d$ such that for every $R > 0$ and for every cube~$K\subset~\mathbb{R}^d$ the inequality
$$
\sum_{i, j\le m}a^{ij}(x,\mu)\xi_i\xi_j\ge \lambda_{K,R}|\xi|^2
$$
holds for some number $\lambda_{K,R}>0$, for all $x\in K$, $\mu \in \mathcal P_{R}(V)$ and $\xi\in\mathbb R^m$.

\vspace*{0.2cm}

{\bf (H1.2)} For every $R > 0$, for every cube $K\subset\mathbb{R}^d$ and for all $i,j$ we have
$$\sup_{x\in K,\, \mu \in \mathcal P_{R}(V)}\left(|a^{ij}(x, \mu)| + |b^i(x, \mu)|\right) < \infty.
$$

{\bf (H1.3)} For every $R > 0$ and for every cube $K = K_y\times K_z\subset\mathbb{R}^d$ there exists a nonnegative continuous and monotone function $\omega_{K,R}$ on $[0, +\infty)$ such that $\omega_{K,R}(0)=0$ and the inequality
$$
\sup_{y \in K_y,\,\mu\in \mathcal P_R(V)}\Bigl(|a^{ij}(y, z, \mu)-a^{ij}(y, z', \mu)|+|b^i(y, z, \mu)-b^i(y, z', \mu)|\Bigr)\le \omega_{K,R}(|z-z'|),
$$
holds for all $i, j$ and for all $z, z'\in K_z$.

\vspace*{0.2cm}

{\bf (H2) Continuity in measures:}

\vspace*{0.2cm}

Here we distinguish between three cases:

\vspace*{0.2cm}

{\bf (The case $m = 0$)} For every $R > 0$, for all $i,j$ and for all $x \in \mathbb R^d$ the $V$-weak convergence of a sequence $\mu_n\in \mathcal P_R(V)$ to a measure $\mu\in \mathcal P_R(V)$ implies the equality
$$
\lim_{n\rightarrow \infty}\Bigl(|a^{ij}(x, \mu_n) - a^{ij}(x,\mu)| + |b^i(x, \mu_n) - b^i(x,\mu)|\Bigr) = 0.
$$

{\bf (The case $m = 1$)} For every $R > 0$, for every cube $K_y \subset\mathbb R^m$, for all $i,j$ and $z \in \mathbb R^{d-m}$ the $V$-weak convergence of a sequence $\mu_n\in \mathcal P_R(V)$ to a measure $\mu\in \mathcal P_R(V)$ implies the equality
$$
\lim_{n\rightarrow\infty}\int_{K_y}\Bigl(|a^{ij}(y,z,\mu) - a^{ij}(y,z,\mu_n)|^{1+\gamma} + |b^i(y,z,\mu) - b^i(y,z,\mu_n)|^{1+\gamma}\Bigr)dy = 0
$$
for some $\gamma > 0$.

\vspace*{0.2cm}

{\bf (The case $m \ge 2$)} For every $R > 0$, for every cube $K_y \subset\mathbb R^m$, for all $i,j$ and $z \in \mathbb R^{d-m}$ the $V$-weak convergence of a sequence $\mu_n\in \mathcal P_R(V)$ to a measure $\mu\in \mathcal P_R(V)$ implies the equality
$$
\lim_{n\rightarrow\infty}\int_{K_y}\Bigl(|a^{ij}(y,z,\mu) - a^{ij}(y,z,\mu_n)|^m + |b^i(y,z,\mu) - b^i(y,z,\mu_n)|^m\Bigr)dy = 0
$$

The case $m = 0$ corresponds to equation (\ref{eq1}), in which the matrix $A$ can be fully degenerate, for example, when $A = 0$. In the case $m = 1$ the matrix $A$ does not degenerate in one of the coordinates, for example, this is the case for $A = \operatorname{diag}(1,0,\ldots,0)$. Similarly, in the case $m \ge 2$, we can take as an example $A = \operatorname{diag}(1,1,\ldots,1,0,\ldots,0)$, where the number one is repeated $m$ times. Let us compare the assumptions for the cases $m = 0$ and $m = 1$ when $d = 1$. In the case $m = 0$ it is necessary to assume (H1.3), which implies uniform convergence in $x$ on every compact set in (H2). In the case $m = 1$ we do not need to assume (H1.3). Moreover, it is sufficient to assume the integral convergence rather than uniform convergence in (H2).

\vspace*{0.2cm}

{\bf (H3) Global conditions:}

\vspace*{0.2cm}

{\bf (H3.1)} There exist positive numbers $C, \Lambda$ and a nonnegative function~$W$ such that the mapping $x\mapsto W(x, \mu)$ belongs to $C^2(\mathbb R^d)$, $\lim_{|x|\rightarrow\infty}W(x,\mu) = +\infty$ and 
$$
	\int_{\mathbb R^d}L_\mu W(x,\mu)\, \mu(dx) \le C - \Lambda \int_{\mathbb R^d}V(x)\,\mu(dx).
$$
for all probability measures $\mu$ with compact support.

\vspace*{0.2cm}

{\bf (H3.2)} There exist positive numbers $C_1, C_2$ and a nonnegative continuous function $H$ such that for all probability measures $\mu$ with compact support one has
$$
	\lim_{|x|\rightarrow\infty}\frac{H(x)}{V(x)} = 0,\qquad -L_\mu W(0,\mu) \le C_1 + C_2\int_{\mathbb R^d} H(x) \mu(dx).
$$

\vspace*{0.2cm}

The following statement is a generalization of \cite[Theorem 1]{ShSh24} for the case of Kolmogorov equations, the coefficients of which depend on the solution. It plays an important role in the proof of Theorem \ref{main_result}. In the case $m = 0$ the statement coincides with \cite[Theorem 5.1]{BS24} when~$\delta=0$.

\begin{theorem}
	\label{cor2}
	Assume that ${\rm (H1)}, {\rm (H2)}$ are fulfilled. Suppose that there exist positive numbers $C$ and $\Lambda$ such that for all $x\in \mathbb R^d$ and $\mu \in \mathcal P(V)$ we have the inequality
	\begin{equation}
		\label{cor2: cond}
		L_\mu V(x) \le C - \Lambda V(x) 
	\end{equation}
	Then equation {\rm (\ref{eq1})} has a probability solution in the set $\mathcal P(V)$.
\end{theorem}

The main result of the paper is the following theorem.

\begin{theorem}
	\label{main_result}
	Assume that ${\rm (H1)}, {\rm (H2)}$ and ${\rm (H3)}$ are fulfilled. Then equation {\rm (\ref{eq1})} has a probability solution in the set $\mathcal P(V)$.
\end{theorem}

Example \ref{ex2} shows that there exist problems satisfying Theorem \ref{main_result} for which the assumptions of Theorem \ref{cor2} are not fulfilled.

\vspace*{0.2cm}

\begin{remark}
	In \cite{ShSh24} an example was constructed showing that a stationary Kolmogorov equation with coefficients discontinuous in $z$ may not have probability solutions at all. Therefore, assumption ${\rm (H1.3)}$ in Theorem \ref{main_result} cannot be weakened. The assumption with a Lyapunov function also cannot be removed: we can consider the equation $\Delta u = 0$, which has no probability solutions.
\end{remark}

\vspace*{0.2cm}

The following corollary is a simple consequence of Theorem \ref{main_result} in the case, when $W = V$. Examples \ref{ex3} and \ref{ex4} show that the class of problems satisfying the assumptions of Theorem~\ref{main_result} is larger than that of Corollary \ref{cor1}.

\begin{corollary}
	\label{cor1}
	Assume that conditions ${\rm (H1)}, {\rm (H2)}$ are fulfilled. Suppose that there exist positive numbers $C_1, C_2, C, \Lambda$ and a nonnegative continuous function $H$ such that 
	\begin{equation}
		\label{cor: cond1}
		\lim_{|x|\rightarrow\infty}\frac{H(x)}{V(x)} = 0, \qquad |a^{ij}(0,\mu)| + |b^i(0, \mu)| \le C_1 + C_2\int_{\mathbb R^d} H(x)\,\mu(dx) 
	\end{equation}
	and 
	\begin{equation}
		\label{cor: cond2}
		\int_{\mathbb R^d}L_\mu V(x)\, \mu(dx) \le C - \Lambda \int_{\mathbb R^d}V(x)\,\mu(dx) 
	\end{equation}
	for all probability measures $\mu$ with compact support. Then equation {\rm (\ref{eq1})} has a probability solution in the set $\mathcal P(V)$.
\end{corollary}

\vspace*{0.2cm}

If we assume that ${\rm (H3.1)}$ holds for all sub-probability measures $\mu$ with compact support, then in Theorem \ref{main_result} assumption ${\rm (H3.2)}$ can be removed. The following theorem gives the corresponding result.
\begin{theorem}
	\label{th2}
	Assume that conditions ${\rm (H1)}, {\rm (H2)}$ are fulfilled. Suppose that ${\rm (H3.1)}$ holds for all sub-probability measures $\mu$ with compact support. Then equation {\rm (\ref{eq1})} has a probability solution in the set $\mathcal P(V)$.
\end{theorem}

\section{Examples}
\label{sec3}

The following example illustrates assumption ${\rm (H2)}$.

\begin{example}
	Let $p\ge 1$. Denote by $\mathcal{P}_p(\mathbb{R}^d)$ the space of probability measures $\mu$ such that $|x|^p\in L^1(\mu)$.
	Recall that the Kantorovich metric $W_p(\mu, \sigma)$ is defined as the infimum
	$$
	\int_{\mathbb{R}^d\times\mathbb{R}^d}|x-y|^p\pi(dxdy)
	$$
	over all probability measures $\pi$ on $\mathbb{R}^d\times\mathbb{R}^d$ with projections $\mu$ and $\sigma$. Consider the space $\mathcal{P}_p(\mathbb{R}^d)$
	with the metric $W_p$.
	
	If $V(x)=\bigl(1+|x|^2\bigr)^{s/2}$, where $s>p$, and $a^{ij}(x,\mu),\, b^i(x,\mu)$ are continuous Borel functions on $\mathbb{R}^d\times \mathcal{P}_p(\mathbb{R}^d)$ for all $i,j$, then assumption ${\rm (H2)}$ is satisfied.
\end{example}

\begin{proof}
	Let $R > 0$. Let us show that for every ball $B\subset\mathbb{R}^d$ the $V$-weak convergence of a sequence $\mu_n\in\mathcal{P}_{R}(V)$ to a measure $\mu\in\mathcal{P}_{R}(V)$
	implies the equalities
	$$
	\lim_{n\to\infty}\sup_{x\in B}\big|a^{ij}(x, \mu^n)-a^{ij}(x, \mu)\big|=0,\qquad \lim_{n\to\infty}\sup_{x\in B}\big|b^i(x, \mu^n)-b^i(x, \mu)\big|=0
	$$
	for all $i,j$.
	
	Let $\overline{B}$ denote the closure of the ball $B$. The set $\mathcal{P}_{R}(V)$ is a compact set in $\mathcal{P}_p(\mathbb{R}^d)$. Therefore, $\overline{B}\times\mathcal{P}_{R}(V)$ is a compact set and
	the mappings $(x, \mu)\mapsto a^{ij}(x, \mu),\, (x, \mu)\mapsto b^i(x, \mu)$ are uniformly continuous on $\overline{B}\times\mathcal{P}_{R}(V)$ for all $i,j$. It remains to note that the $V$-weak convergence of a sequence $\mu_n\in\mathcal{P}_{R}(V)$ to a measure $\mu\in\mathcal{P}_{R}(V)$
	implies the equality $\lim_{n\to\infty}W_p(\mu_n, \mu)=0$.
\end{proof}

Let us consider the example illustrating assumptions ${\rm (H1)}$ and ${\rm (H3)}$.
\begin{example}
	\label{ex1}
	{\rm Let $V(x)=(1+|x|^2)^{s/2}$, $H(x)=|x|^p$, where $0\le p<s$. Suppose that for every cube $K \subset \mathbb R^d$ there exists a continuous function $\Psi_K$ such that
	
	(i) $\displaystyle{|a^{ij}(x,\mu) - a^{ij}(x',\mu)| + |b^i(x,\mu) - b^i(x',\mu)| \le \Psi_K\left(\int_{\mathbb R^d} |y|^s\mu(dy)\right) |x-x'|.}$\\
	for all $x,x' \in K$ and all probability measures $\mu$ with compact support. Let us also assume that there exist positive numbers $C,C_1,C_2,\delta$ and a continuous function $\beta$ such that $\beta(0) = 0$ and
	
(ii) $\displaystyle{|a^{ij}(x,\mu)| \le C\left(1 +|x|^p + \int_{\mathbb R^d}|y|^p\mu(dy)\right),}\qquad\displaystyle{|b^i(x,\mu)| \le C\left(1 + \beta(x) + \int_{\mathbb R^d}|y|^p\mu(dy)\right),}$

(iii) $\displaystyle{\int_{\mathbb R^d} |y|^{s-2}\langle b(y,\mu),  y\rangle\mu(dy) \le C_1-C_2\int_{\mathbb R^d}|y|^{s+p-1+\delta}\mu(dy)}$\\
for all $i,j$, for all $x,x'$ and all probability measures $\mu$ with compact support. If, in addition, ${\rm (H2)}$ holds, then the assumptions of Corollary \ref{cor1} are satisfied for $m = 0$.
}
\end{example}

\begin{proof}
	Let us first verify that ${\rm (H3.1)}$ is fulfilled. It is enough to verify this for $V(x) = |x|^s$. Using the equality
	$$
	L_{\mu}V(x)=s|x|^{s-2}{\rm trace}\, A(x, \mu)+s(s-2)|x|^{s-4}\langle A(x, \mu)x, x\rangle+
	s|x|^{s-2}\langle b(x, \mu), x\rangle,
	$$
	we obtain the estimate
	\begin{multline*}
		\int_{\mathbb R^d} L_{\mu}V(y)\mu(dy)\le\\ 
		s(s+d-2)C\int_{\mathbb R^d}|y|^{s-2}\left(1 +|y|^p + \int_{\mathbb R^d}|z|^p\mu(dz)\right)\mu(dy) + s\int_{\mathbb R^d}|y|^{s-2}\langle b(y, \mu), y\rangle\mu(dy).
	\end{multline*}
Let $B_1$ denote the ball with center at zero and radius one. The inequality
$$
\int_{\mathbb R^d} |y|^r\mu(dy) = \int_{B_1} |y|^r\mu(dy) + \int_{\mathbb R^d\setminus B_1} |y|^r\mu(dy) \le 1 + \int_{\mathbb R^d} |y|^q\mu(dy).
$$
holds for all $0 \le r \le q$. Then we have
$$
\int_{\mathbb R^d} L_{\mu}V(y)\mu(dy)\le N_1 + N_2 \int_{\mathbb R^d}|y|^{s + p-2}\mu(dy)+ s\int_{\mathbb R^d}|y|^{s-2}\langle b(y, \mu), y\rangle\mu(dy),
$$
where the numbers $N_1, N_2$ do not depend on $\mu$. Let us choose a number $R > 0$ such that $s\,C_2R^\delta = N_2 + 1$. Then we obtain the estimate
\begin{multline*}
	\int_{\mathbb R^d}|y|^{s-2}\langle b(y, \mu), y\rangle\mu(dy) \le 
	C_1 - C_2R^{\delta}\int_{\mathbb R^d\setminus B_R}|y|^{s+p-1}\mu(dy) =\\ 
	C_1 - \frac{N_2+1}{s}\int_{\mathbb R^d\setminus B_R}|y|^{s+p-1}\mu(dy) \le C_1 - \frac{N_2+1}{s}\int_{\mathbb R^d}|y|^{s+p-1}\mu(dy) + C_2R^{s+p-1+\delta}.
\end{multline*}
Therefore,
$$
\int_{\mathbb R^d} L_{\mu}V(y)\mu(dy)\le N - \int_{\mathbb R^d}|y|^{s+p-1}\mu(dy) \le N + 1 - \int_{\mathbb R^d}|y|^{s}\mu(dy),
$$
where $N = N_1 + C_1 + sC_2R^{s+p-1+\delta}$. Thus, ${\rm (H3.1)}$ is satisfied.

Let us verify that the remaining conditions hold true. (H1.1) is fulfilled for $m = 0$. By (i) assumptions (H1.2) and (H1.3) are satisfied. Condition (\ref{cor: cond1}) is valid by (ii).
\end{proof}

\vspace*{0.2cm}

The following example shows that the integral condition may hold even when the pointwise condition fails.

\begin{example}
	\label{ex2}
	{\rm The nonlinear Kolmogorov equation (\ref{eq1}) on $\mathbb R$ with the coefficients
	$$
	a(x,\mu) = x^2\left(\int_\mathbb R |y|\,\mu(dy)\right)^3,\quad b(x,\mu) = -2x^3\left(\int_\mathbb R |y|\,\mu(dy)\right)
	$$	
	satisfies the assumptions of Corollary \ref{cor1}, but the assumptions of Theorem \ref{cor2} are not satisfied for any Lyapunov function $V$.
}
\end{example}
\begin{proof}
	Let $V(x) = x^2/2$. Then the equality
	$$
	L_\mu V = x^2\left(\int_\mathbb R |y|\,\mu(dy)\right)^3 - 2x^4 \left(\int_\mathbb R |y|\,\mu(dy)\right).
	$$
	holds for all $x \in \mathbb R$ and all probability measures $\mu$ with compact support. Integrating the equality, we obtain
	\begin{multline*}
		\int_{\mathbb R}L_\mu Vd\mu = \left(\int_\mathbb R |y|\,\mu(dy)\right)\left[\left(\int_{\mathbb R} y^2\mu(dy)\right)\left(\int_\mathbb R |y|\,\mu(dy)\right)^2 - 2\left(\int_{\mathbb R}y^4 \mu(dy)\right)\right] \le\\ 
		-\int_{\mathbb R}y^4 \mu(dy)\int_\mathbb R |y|\,\mu(dy),
	\end{multline*}
where the last inequality is satisfied by H\"older's inequality. On the one hand,
	$$
		\int_\mathbb R |y|\,\mu(dy)\int_{\mathbb R}y^4 \mu(dy) \ge  \int_{|y|>1}y^4 \mu(dy)\int_{|y|>1}|y|\,\mu(dy).
	$$
On the other hand,
$$
\int_{|y| > 1} y^2 \mu(dy) \le \left(\int_{|y| > 1} y^4 \mu(dy)\right)^{1/2}\mu(|y|>1)^{1/2} \le \left(\int_{|y| > 1} y^4 \mu(dy)\right)^{1/2}\left(\int_{|y| > 1} |y| \mu(dy)\right)^{1/2}.
$$
Then we have
	\begin{multline*}
	\int_{\mathbb R}y^4 \mu(dy)\int_\mathbb R |y|\,\mu(dy) \ge  \left(\int_{|y| > 1} y^2 \mu(dy)\right)^2 = \left(\int_\mathbb R y^2 \mu(dy) - \int_{|y| \le 1} y^2 \mu(dy)\right)^2 \ge\\ \left(\int_\mathbb R y^2 \mu(dy)\right)^2 - 2 \int_\mathbb R y^2 \mu(dy) \int_{|y| \le 1} y^2 \mu(dy) \ge \left(\int_\mathbb R y^2 \mu(dy)\right)^2 - 2 \int_\mathbb R y^2 \mu(dy) \ge \int_\mathbb R y^2 \mu(dy)-3.
\end{multline*}
Thus, (\ref{cor: cond2}) is fulfilled.

We have $a(0,\mu) = b(0,\mu) = 0$. Therefore, we can take zero function as the function $H$ in (\ref{cor: cond1}). 

Let us verify that ${\rm (H1)},\, {\rm (H2)}$ are satisfied. ${\rm (H1.1)}$ is fulfilled for $m = 0$. For the function $V(x) = x^2/2$ the set $\mathcal P_{R}(V)$, introduced in the previous section, is the set of all probability measures $\mu$ on $\mathbb R$, satisfying the condition
$$
\int_\mathbb R y^2 \mu(dy) \le 2R,
$$
then ${\rm (H1.2)}$ holds. Assumption ${\rm (H1.3)}$ is satisfied, since for any every $K \subset \mathbb R$ the inequality
	$$
|a(x,\mu) - a(x',\mu)| + |b(x,\mu) - b(x',\mu)| \le C(K, R)|x - x'|,
$$
holds for all $x, x' \in K$ and all $\mu\in \mathcal P_R(V)$, where
$$
C(K, R) = 16R^3\max_{x\in K}|x| + 16R\max_{x\in K}|x|^2.
$$

In the definition of the $V$-weak convergence we can substitute the function $\zeta(x) = |x|$. Therefore, ${\rm (H2)}$ is also satisfied. 

Let us show that the assumptions of Corollary \ref{cor2} are not fulfilled. Suppose that there exist positive numbers $C, \Lambda$ and a function $V\in C^2(\mathbb R)$ such that $V\ge 0,\,\lim_{|x|\rightarrow\infty}V(x) = +\infty$ and
$$
L_\mu V(x) \le C - \Lambda V(x)
$$
for all $x \in \mathbb R$ and all $\mu \in \mathcal P(V)$. Substituting $\mu = \delta_0$, we obtain
$$
0 \le C - \Lambda V(x).
$$
Letting $|x| \rightarrow \infty$ to infinity, we obtain a contradiction.
\end{proof}

The following example shows that allowing the Lyapunov function to depend on the measure enlarges the class of problems covered by the integral condition.

\begin{example}
	\label{ex3}
	{\rm Let $a(x)$ be a Lipschitz function with compact support. The nonlinear Kolmogorov equation (\ref{eq1}) on $\mathbb R$ with the coefficients
	$$
	a(x,\mu) = a(x),\quad b(x,\mu) = \int_\mathbb R (2y-x)\mu(dy)
	$$
	satisfies the assumptions of Theorem \ref{main_result}, but the assumptions of Corollary \ref{cor1} are not satisfied for any Lyapunov function $V$.}
\end{example}
\begin{proof}
	For every probability measure $\mu$ with compact support we set
	$$
	W(x, \mu) = \frac{1}{2}\int_\mathbb R (x - 2y)^2\mu(dy).
	$$
	We have the equality
		$$
	L_\mu W(x,\mu) = a(x) - \left(\int_\mathbb R (x - 2y)\mu(dy)\right)^2.
	$$
	Integrating the equality, we obtain
		\begin{multline*}
		\int_\mathbb R L_\mu W(x,\mu) \mu(dx) = \int_{\mathbb R} a(x)\mu(dx) - \int_\mathbb R \left(\int_\mathbb R (x - 2y)\mu(dy)\right)\left(\int_\mathbb R (x - 2z)\mu(dz)\right)\mu(dx) =\\
		\int_{\mathbb R} a(y)\mu(dy) - \iiint_{\mathbb R^3}(x - 2y) (x - 2z)\mu(dy)\mu(dz)\mu(dx).
	\end{multline*}
	Note that
	\begin{multline*}
	\iiint_{\mathbb R^3}(x - 2y) (x - 2z)\mu(dy)\mu(dz)\mu(dx) =\\ 
	\iiint_{\mathbb R^3}(x^2 - 2xy - 2xz + 4yz)\mu(dy)\mu(dz)\mu(dx) = 
	\int_\mathbb R x^2\mu(dx)
	\end{multline*}
	Then we get the estimate
	$$
	\int_\mathbb R L_\mu W(x,\mu) \mu(dx) \le \max_{x\in\mathbb R}|a(x)| - \int_\mathbb R x^2\mu(dx)
	$$
	and we can take $V(x) = x^2$ as the function in ${\rm (H3.1)}$.
	
	Let us verify that ${\rm (H1)},\, {\rm (H2)}$ are satisfied. ${\rm (H1.1)}$ is fulfilled for $m = 0$. For the function $V(x) = x^2$ the set $\mathcal P_{R}(V)$, introduced in the previous section, is the set of all probability measures~$\mu$ on $\mathbb R$, satisfying the condition
	$$
	\int_{\mathbb R}y^2\mu(dy) \le 2R,
	$$
	then ${\rm (H1.2)}$ holds. Assumption ${\rm (H1.3)}$ is satisfied, since $a(x)$ is a Lipschitz function and for any cube $K \subset \mathbb R$ the equality
	$$
	|b(x,\mu) - b(x',\mu)| = |x - x'|.
	$$
	holds for all $x, x' \in K$ and all $\mu\in \mathcal P_R(V)$.
	
	In the definition of the $V$-weak convergence we can substitute the function $\zeta(x) = x$. Therefore, ${\rm (H2)}$ is also satisfied. 
	
	Let us show that the assumptions of Corollary \ref{cor1} are not satisfied. Suppose that there exist positive numbers $C, \Lambda$ and a function $V\in C^2(\mathbb R)$ such that $V\ge 0, \lim_{|x|\rightarrow\infty}V(x) = +\infty$ and
	$$
	\int_\mathbb R L_\mu V(x)\mu(dx) \le C - \Lambda \int_\mathbb R V(x) \mu(dx)
	$$
	for any probability measure $\mu$ with compact support. Then we have
	$$
	\int_\mathbb R L_\mu V(x)d\mu = \int_{\mathbb R}a(x)V''(x)\mu(dx) - \iint_{\mathbb R^2} (x - 2y) V'(x)\mu(dy)\mu(dx) \le C - \Lambda \int_\mathbb R V(x) d\mu
	$$
	Substituting $\mu = \delta_x$ into this inequality, we get
	\begin{equation}
		\label{exmp: eq}
		a(x)V''(x) + xV'(x) \le C - \Lambda V(x).
	\end{equation}
	Since $a(x)$ is a function with compact support, there exists a number $r > 0$ such that $a(x) = 0$ for all $|x| > r$. Besides, we can find a number $r' > r$ such that $C - \Lambda V(x) \le 0$ whenever $|x| > r'$. According to (\ref{exmp: eq}), we have $x V'(x) \le 0$ for all $|x| > r'$. But then $V$ does not increase for all $x > r'$. This contradicts the assumption $\lim_{|x|\rightarrow\infty}V(x) = +\infty$.
\end{proof}

\vspace*{0.2cm}

The following example shows that the application of a measure-dependent Lyapunov function also expands the class of problems, to which the integral condition can be applied, in the case of a diffusion coefficient $A$ with non-compact support.

\begin{example}
	\label{ex4}
	{\rm The nonlinear Kolmogorov equation on $\mathbb R$ with the coefficients
	$$
	a(x,\mu) = xI_{\{x\ge 0\}},\quad b(x,\mu) = \int_\mathbb R (2y-x)\mu(dy),
	$$
	where $I_{\{x\ge 0\}}$ is the indicator function of the set $\{x \ge 0\}$, satisfies Theorem \ref{main_result}, but the assumptions of Corollary \ref{main_result} are not satisfied for any Lyapunov function $V$.}
\end{example}

\begin{proof}
	For every probability measure $\mu$ with compact support we set	
	$$
	W(x, \mu) = \frac{1}{2}\int_\mathbb R (x - 2y)^2\mu(dy).
	$$
	We have the equality
	$$
	L_\mu W(x,\mu) = x I_{\{x\ge0\}} - \left(\int_\mathbb R (x - 2y)\mu(dy)\right)^2.
	$$
	Similar to the previous example, we obtain
	$$
		\int_\mathbb R L_\mu V(x,\mu) d\mu = \int_0^\infty x\,\mu(dx) - \int_\mathbb R x^2\mu(dx) \le
		\left(\int_{\mathbb R} x^2\mu(dx)\right)^{1/2} - \int_\mathbb R x^2\mu(dx),
	$$
	Using the inequality $ab\le (a^2 + b^2)/2$, we obtain the estimate
	$$
	\int_\mathbb R L_\mu V(x,\mu) d\mu \le \frac{1}{2} - \frac{1}{2}\int_\mathbb R x^2\mu(dx).
	$$
	Then we can take $V(x) = x^2$ as the function in ${\rm (H3.1)}$. Similar to the previous example, ${\rm (H1)},\, {\rm (H2)}$ are satisfied.
	
		Let us show that the assumptions of Corollary \ref{cor1} are not satisfied. Suppose that there exist positive numbers $C, \Lambda$ and a function $V\in C^2(\mathbb R)$ such that $V\ge 0, \lim_{|x|\rightarrow\infty}V(x) = +\infty$ and
	$$
	\int_\mathbb R L_\mu V(x)\mu(dx) \le C - \Lambda \int_\mathbb R V(x) \mu(dx)
	$$
	for any probability measure $\mu$ with compact support. Then we have
	$$
	\int_\mathbb R L_\mu V(x)d\mu = \int_{\mathbb R}xI_{\{x\ge 0\}}V''(x)\mu(dx) - \iint_{\mathbb R^2} (x - 2y) V'(x)\mu(dy)\mu(dx) \le C - \Lambda \int_\mathbb R V(x) d\mu.
	$$
	Substituting $\mu = \delta_x$ into this inequality, we get
	\begin{equation}
		\label{exmp: eq2}
		x\left(I_{\{x\ge 0\}}V''(x) + V'(x)\right) \le C - \Lambda V(x).
	\end{equation}
	There exists a number $r > 0$ such that $C - \Lambda V(x) \le 0$ for all $|x| > r$. According to (\ref{exmp: eq2}), we have $V''(x) + V'(x) \le 0$ for all $x > r$, that is, $V'' + V'$ does not increase for all $x > r$. Therefore, there exists a number $C$ such that $V'(x) + V(x) \le C$ whenever $|x| > r$. But then $\lim_{x \rightarrow +\infty}V'(x) = -\infty$. This contradicts the assumption $\lim_{|x|\rightarrow\infty}V(x) = +\infty$. 
\end{proof}

\section{Auxiliary results}
\label{sec4}
This section is devoted to the assertions, playing the crucial role in the proof of Theorem~\ref{main_result}. Further we write $L^*_\sigma \mu = 0$ in the case when $\mu$ is a probability solution to equation (\ref{eq1}) with the measure $\sigma$ in the coefficients. The following result is a generalization of \cite[Lemma~1]{ShSh24}. 

As in Section \ref{sec2} we assume that the upper left $m\times m$ corner of the matrix $A(x,\mu)$ is non-degenerate for all probability measures $\mu$ and $x \in \mathbb R^d$. Let us split the coordinates of variables into two parts $x = (y, z)$, where $y = (x_1, \ldots, x_m)$ and $z = (x_{m+1}, \ldots, x_d)$.
\begin{lemma}
	\label{lemma: regular}
	Assume that conditions ${\rm (H1)}, {\rm (H2)}$ are fulfilled. Let $R > 0$, $\sigma \in \mathcal P_R(V)$ and $\mu$ be a probability solution to the equation $L^*_\sigma \mu = 0$. \\
	(i) Then for every function $\eta\in C_0^{\infty}(\mathbb{R}^{d-m})$, $0\le\eta\le 1$, the projection of the measure
	$\eta\mu$ onto~$\mathbb{R}^m_y$ has a density $\varrho$ with respect to the Lebesgue measure, and for every cube $K_y \subset \mathbb{R}^m_y$ the density $\varrho$ belongs to $L^{r}(K_y)$, where $r=m'$ for $m>1$ and $r$ is an arbitrary number greater than one for $m=1$. \\
	(ii) In addition, if an open cube $Q_y$ contains the closure of the cube $K_y$, an open cube $Q_z$ contains the support of $\eta$ and $Q = Q_y\times Q_z$, then we have the estimate
	$$\|\varrho\|_{L^{r}(K_y)}\le C,$$
	where the number $C$ does not depend on $\mu,\sigma$ and depends only on $R$, $K_y$, $Q$, $\eta$, $\lambda_{Q,R}$, $m$, $d$, $\omega_{Q,R}$ and 
	$C_{Q,R} = \sup_{x\in Q,\, \nu \in \mathcal P_{R}(V)}\left(|a^{ij}(x, \nu)| + |b^i(x, \nu)|\right).
	$
\end{lemma}
\begin{proof}
	Let $\mu$ be a probability solution to the equation $L^*_\sigma \mu = 0$. We introduce the functions
	$$
	\alpha^{ij}(x) = a^{ij}(x,\sigma),\quad \beta^i(x) = b^i(x,\sigma),
	$$
	which are defined for all $x \in \mathbb R^d$. Then the measure $\mu$ is a solution to the equation 
	$$
	\sum_{i,j = 1}^d \partial_{x_i}\partial_{x_j}\bigl(\alpha^{ij}(x)\mu\bigr)-\sum_{i=1}^d\partial_{x_i}\bigl(\beta^i(x)\mu\bigr)=0.
	$$
	By ${\rm (H1)}$ for all $x\in Q$, $\xi\in\mathbb R^m$ and all $i,j$ we have
	$$
	\sum_{i, j\le m}\alpha^{ij}(x)\xi_i\xi_j\ge \lambda_{Q,R}|\xi|^2,\quad |\alpha ^{ij}(x)| + |\beta^i(x)| \le C_{Q,R},
	$$
	and for all $y \in Q_y, z,z'\in Q_z$ we have
	$$
	|\alpha^{ij}(y, z)-\alpha^{ij}(y, z')|+|\beta^i(y, z)-\beta^i(y, z')|\le \omega_{K,R}(|z-z'|).
	$$
	Therefore, the statement of the lemma follows from \cite[Lemma 1]{ShSh24}.
\end{proof}

\vspace*{0.2cm}

\begin{remark}
	\label{rm: reg_with_poten}
	If we raplace the operator $L_\sigma$ by $L_\sigma + f(x)$, where $f$ is a continuous function, the statement of Lemma \ref{lemma: regular} remains true.
\end{remark}
\begin{proof}
	 We say that a probability measure $\mu$ is a solution to the stationary Kolmogorov equation with the operator $L_\sigma + f(x)$ if the identity
	 \begin{equation}\label{integr2}
	 	\int_{\mathbb{R}^d}L_\sigma u(x)\mu(dx) + \int_{\mathbb R^d} f(x)u(x)\mu(dx)=0.
	 \end{equation}
	 holds for every function $u\in C_0^{\infty}(\mathbb{R}^d)$. We substitute the function 
	 $$u(y, z)=\psi(y)\eta(z)$$ 
	 into identity (\ref{integr2}), where $\psi\in C_0^{\infty}(Q_y)$. 	 
	 Then by ${\rm (H1)}$ we obtain the estimate
	\begin{multline*}
	\int_{\mathbb R^d}\Biggl(\sum_{i, j\le m}a^{ij}(y, z, \sigma)\partial_{y_i}\partial_{y_j}\psi(y)\Biggr)\eta(z)\mu(dydz)\le\\
	N\Bigl(\sup_y|\psi(y)|+\sup_y|\nabla\psi(y)|\Bigr) - \int_{\mathbb R^d}f(y,z)\psi(y)\eta(z)\mu(dydz),
	\end{multline*}
	where $N$ does not depend on $\mu$ and $\sigma$, since $C_{Q,R} = \sup_{x\in Q,\, \nu \in \mathcal P_{R}(V)}\left(|a^{ij}(x, \nu)| + |b^i(x, \nu)|\right)$ does not depend on $\mu$ and $\sigma$. Note that
	$$
	\int_{\mathbb R^d}f(y,z)\psi(y)\eta(z)\mu(dydz) \le C(f)\sup_y|\psi(y)|,
	$$
	where $C(f) = \sup_{x \in Q}|f(x)|\sup_z|\eta(z)|$. Denoting $N' = N + C(f)$, we get the inequality
	$$
		\int_{\mathbb R^d}\Biggl(\sum_{i, j\le m}a^{ij}(y, z, \sigma)\partial_{y_i}\partial_{y_j}\psi(y)\Biggr)\eta(z)\mu(dydz)\le\\
		N'\Bigl(\sup_y|\psi(y)|+\sup_y|\nabla\psi(y)|\Bigr).
	$$
	Since the $\lambda_{Q,R}$ does not depend on $\mu$ and $\sigma$, we can repeat the arguments of \cite[Lemma 1]{ShSh24} and obtain the required estimate.
\end{proof}

\vspace*{0.2cm}

Let us consider an arbitrary cube $K = K_y \times K_z \subset \mathbb R^d$ and a function $\eta\in C_0^{\infty}(\mathbb{R}^{d-m})$ such that $0\le\eta\le 1$ and $\eta = 1$ on $K_z$. We denote by $\mathcal M_{K,S}(\eta)$ the set of all probability measures~$\mu$ on~$\mathbb R^d$ such that the projection of the measure
$\eta\mu_n$ onto $\mathbb{R}^m_y$ has a density $\varrho$ with respect to the Lebesgue measure and
$$
\|\varrho\|_{L^{r}(K_y)}\le S,
$$
where $r=m'$ for $m>1$ and $r$ is an arbitrary number greater than one for $m=1$.

\begin{lemma}
	\label{lemma: converg}
	Assume that conditions ${\rm (H1)}, {\rm (H2)}$ are fulfilled. Then for all $i,j$ the weak convergence of a sequence $\sigma_n\in \mathcal P_R(V)$ to a measure $\sigma\in \mathcal P_R(V)$ implies the equality
	$$
	\lim_{n \rightarrow \infty}\sup_{\mu \in \mathcal M_{K,S}(\eta)}\int_K\Bigl(|a^{ij}(x, \sigma_n) - a^{ij}(x, \sigma)| + |b^i(x, \sigma_n) - b^i(x, \sigma)|\Bigr)d\mu = 0.
	$$
\end{lemma}
\begin{proof}
	Fix $\varepsilon > 0$. Let $\mu \in \mathcal M_{K,S}(\eta)$. We cut the cube $K_z$ into $N$ pairwise disjoint Borel sets~$\Delta_k$ such that for every $k$ the inequality $\omega_{K,R}(|z-z'|)\le\varepsilon$ holds for all $z, z'\in \Delta_k$. In each $\Delta_k$ we choose an arbitrary point~$\xi_k$. For all $i,j$ we have the estimate
	\begin{multline}
		\label{estim_1}
		\int_{K_y\times K_z}|a^{ij}(y,z,\sigma) - a^{ij}(y,z,\sigma_n)|\,\mu(dydz) = \sum_{k = 1}^{N} \int_{K_y\times \Delta_k}|a^{ij}(y,z,\sigma) - a^{ij}(y,z,\sigma_n)|\,\mu(dydz) \le\\
		\sum_{k = 1}^{N} \int_{K_y\times \Delta_k}|a^{ij}(y,z,\sigma) - a^{ij}(y,\xi_k,\sigma)|\,\mu(dydz) + \sum_{k = 1}^{N} \int_{K_y\times \Delta_k}|a^{ij}(y,\xi_k,\sigma_n) - a^{ij}(y,z,\sigma_n)|\,\mu(dydz) +\\
		\sum_{k = 1}^{N} \int_{K_y\times \Delta_k}|a^{ij}(y,\xi_k,\sigma) - a^{ij}(y,\xi_k,\sigma_n)|\,\mu(dydz).
	\end{multline}
By ${\rm (H1.3)}$ the inequality 
	$$|a^{ij}(y,z,\sigma) - a^{ij}(y,\xi_k,\sigma)| + |a^{ij}(y,z,\sigma_n) - a^{ij}(y,\xi_k,\sigma_n)| < 2\varepsilon.$$
holds for all $z \in \Delta_k$, $y \in K_y$. Therefore, the first and second sums in inequality (\ref{estim_1}) are less than $\varepsilon$. Let us estimate the last sum in inequality (\ref{estim_1}). To do this, we consider two cases: $m = 0$ and $m > 0$. 

If $m = 0$, then there are no variables $y$ and the last sum in inequality (\ref{estim_1}) has the form
$$
\sum_{k = 1}^{N} \int_{K_y\times \Delta_k}|a^{ij}(\xi_k,\sigma) - a^{ij}(\xi_k,\sigma_n)|\,\mu(dydz) = \sum_{k = 1}^{N} |a^{ij}(\xi_k,\sigma) - a^{ij}(\xi_k,\sigma_n)|\,\mu(K_y\times \Delta_k).
$$
By ${\rm (H2)}$ there exists a number $N'$ such that
$$
|a^{ij}(\xi_k,\sigma) - a^{ij}_\delta(\xi_k,\sigma_n)| < \varepsilon
$$
for all $n \ge N'$ and for all $k \in \{1,\ldots, N\}$. Hence the last sum in inequality (\ref{estim_1}) is less than~$\varepsilon$.

Let us now consider the case $m > 0$. Let $\eta \in C_0^\infty(\mathbb R^{d-m})$, $0 \le \eta \le 1$ and $\eta(z) = 1$ for all $z \in K_z$. Then the inequality
$$
\int_{K_y\times \Delta_k}|a^{ij}(y,\xi_k,\sigma) - a^{ij}(y,\xi_k,\sigma_n)|\,\mu(dydz) \le \int_{K_y\times \mathbb R_z^{d-m}}|a^{ij}(y,\xi_k,\sigma) - a^{ij}(y,\xi_k,\sigma_n)|\,\eta(z)\mu(dydz)
$$
holds for all $k\in \{1,\ldots, N\}$. Since $\mu \in \mathcal M_{K,S}(\eta)$, the projection of the measure $\eta\mu$ onto $\mathbb{R}^m_y$ has a density $\varrho$ and $\|\varrho\|_{L^{r}(K_y)} \le S$. By H\"older's inequality we have
\begin{multline*}
	\int_{K_y}|a^{ij}(y,\xi_k,\sigma) - a^{ij}(y,\xi_k,\sigma_n)|\rho(y)dy \le
	S\left(\int_{K_y}|a^{ij}(y,\xi_k,\sigma) - a^{ij}(y,\xi_k,\sigma_n)|^{r'}dy\right)^{1/r'}.
\end{multline*}
Here $r' = m$ for $m > 1$ and $r'$ is an arbitrary number greater than one for $m=1$. In the case $m = 1$ we put $r' = 1 + \gamma$, where $\gamma$ is the number from assumption ${\rm (H2)}$. By ${\rm (H2)}$ there exists a number $N'$ such that
$$
\left(\int_{K_y}|a^{ij}(y,\xi_k,\sigma) - a^{ij}(y,\xi_k,\sigma_n)|^{r'}dy\right)^{1/r'} < \frac{\varepsilon}{SN}
$$
for all $n \ge N'$ and for all $k \in \{1,\ldots, N\}$. Therefore, the last sum in inequality (\ref{estim_1}) is less than~$\varepsilon$.

Thus, we arrive to the estimate
$$
\int_{K_y\times K_z}|a^{ij}(y,z,\sigma) - a^{ij}(y,z,\sigma_n)|\,\mu(dydz) < 3\varepsilon.
$$
for all $n \ge N'$. The required estimate for the coefficients $b^i$ is obtained in the same way. This completes the proof.
\end{proof}

\vspace*{0.2cm}

In the proof of Theorem \ref{main_result} we will use the following procedure for approximating the coefficients by continuous functions. 

Let $\delta \in (0,1)$, $y \in \mathbb R^m_y$ and $h_\delta(y) = \delta^{-m}h(y/\delta)$, where $h$ is a nonnegative function such that the support of $h$ is in $\{y: |y| < 1\}$ and $\|h\|_{L^1(\mathbb R^m)} = 1$. We introduce the functions
$$
	a^{ij}_{\delta}(y, z, \sigma)=\int_{\mathbb R^m_y} h_{\delta}(y-v)a^{ij}(v, z, \sigma)\,dv, \quad b^{i}_{\delta}(y, z, \sigma)=\int_{\mathbb R^m_y} h_{\delta}(y-v)b^{i}(v, z, \sigma)\,dv.
$$

The following lemma on approximation holds.
\begin{lemma}
	\label{lemma: approx}
	Assume that conditions ${\rm (H1)}, {\rm (H2)}$ are fulfilled. Let the cube $Q=Q_y\times Q_z$ contain the closure of the cube $K$ and the distance
	from the boundary of the cube $Q_y$ to the boundary of the cube $K_y$ be greater than one. Then we have the equality
$$
\lim_{\delta \rightarrow 0}\sup_{\mu \in \mathcal M_{Q,S}(\eta),\,\sigma \in \mathcal P_R(V)}\int_K\Bigl(|a^{ij}(x, \sigma) - a^{ij}_\delta(x, \sigma)| + |b^i(x, \sigma) - b^i_\delta(x, \sigma)|\Bigr)d\mu = 0.
$$	
\end{lemma}

\begin{proof}
	Fix $\varepsilon > 0$. Let $\mu \in \mathcal M_{Q,S}(\eta)$ and $\sigma \in \mathcal P_R(V)$. We cut the cube $K_z$ into $N$ pairwise disjoint Borel sets $\Delta_k$ such that for every $k$ the inequality $\omega_{K,R}(|z-z'|)\le\varepsilon$ holds for all $z, z'\in \Delta_k$. In each set $\Delta_k$ we choose an arbitrary point~$\xi_k$. Then for all $i,j$ we have the estimate
	\begin{multline*}
		\int_{K_y\times K_z}|a^{ij}(y,z, \sigma) - a^{ij}_\delta(y,z, \sigma)|\,\mu(dydz) = \sum_{k = 1}^N \int_{K_y\times \Delta_k}|a^{ij}(y,z, \sigma) - a^{ij}_\delta(y,z, \sigma)|\,\mu(dydz) \le\\ 
		2\varepsilon + \sum_{k = 1}^N \int_{K_y\times \Delta_k}|a^{ij}(y,\xi_k, \sigma) - a^{ij}_\delta(y,\xi_k, \sigma)|\mu(dydz)
	\end{multline*}
	If $m = 0$, then there are no variables $y$ and the last sum has the form
	$$
	\sum_{k = 1}^{N} \int_{K_y\times \Delta_k}|a^{ij}(\xi_k,\sigma) - a^{ij}_\delta(\xi_k,\sigma)|\,\mu(dydz) = \sum_{k = 1}^{N} |a^{ij}(\xi_k,\sigma) - a^{ij}_\delta(\xi_k,\sigma)|\,\mu(K_y\times \Delta_k),
	$$
	which is less than $\varepsilon$ for all $k \in \{1,\ldots,N\}$ and all $\sigma \in \mathcal P_R(V)$. Then the right-hand side of the previous inequality is less than $3\varepsilon$.
	
	Now we consider the case $m > 0$. Let $\eta \in C_0^\infty(\mathbb R^{d-m})$, $0 \le \eta \le 1$ and $\eta(z) = 1$ for all $z \in Q_z$. For all $i,j$ we have
	\begin{multline*}
		\int_{K_y\times K_z}|a^{ij}(y,z, \sigma) - a^{ij}_\delta(y,z, \sigma)|\,\mu(dydz) = \sum_{k = 1}^N \int_{K_y\times \Delta_k}|a^{ij}(x, \sigma) - a^{ij}_\delta(x, \sigma)|\,\mu(dydz) \le\\ 
		2\varepsilon + \sum_{k = 1}^N \int_{K_y\times \mathbb R^{d-m}_z}|a^{ij}(y,\xi_k, \sigma) - a^{ij}_\delta(y,\xi_k, \sigma)|\eta(z)\mu(dydz)
	\end{multline*}
Since $\mu \in \mathcal M_{Q,S}(\eta)$, the projection of the measure $\eta\mu_n$ onto $\mathbb{R}^m_y$ has a density $\varrho$ and we have the esimate $\|\varrho\|_{L^{r}(K_y)} \le~S$. Using Hölder's inequality, we obtain
$$
\int_{K_y}|a^{ij}(y,\xi_k, \sigma) - a^{ij}_\delta(y,\xi_k, \sigma)|\rho(y)dy \le S\left(\int_{K_y}|a^{ij}(y,\xi_k, \sigma) - a^{ij}_\delta(y,\xi_k, \sigma)|^{r'}\,dy\right)^{1/r'}.
$$
Here $r' = m$ for $m > 1$ and $r'$ is an arbitrary number greater than one for $m=1$. In the case $m = 1$ we put $r' = 1 + \gamma$, where $\gamma$ is the number from assumption ${\rm (H2)}$.
By ${\rm (H2)}$ the mapping 
$$
\sigma \longmapsto\int_{K_y}|a^{ij}(y,\xi_k,\mu) - a^{ij}(y, \xi_k, \sigma)|^{r'}dy
$$
is continuous on the compact $\mathcal P_R(V)$ for all $\mu \in \mathcal P_R(V)$ and $k \in \{1,\ldots, N\}$. Therefore, we can find $\sigma_1,\ldots, \sigma_M \in \mathcal P_R(V)$ such that 
$$
\inf_{n \in \{1,\ldots, M\}}\left(\int_{K_y}|a^{ij}(y,\xi_k,\sigma) - a^{ij}(y, \xi_k, \sigma_n)|^{r'}dy\right)^{1/r'} < \frac{\varepsilon}{S}
$$
for all $\sigma \in \mathcal P_R(V)$. In addition, 
$$
\left(\int_{K_y}|a^{ij}_\delta(y,\xi_k,\sigma) - a^{ij}_\delta(y, \xi_k, \sigma_n)|^{r'}dy\right)^{1/r'} \le \left(\int_{K_y}|a^{ij}(y,\xi_k,\sigma) - a^{ij}(y, \xi_k, \sigma_n)|^{r'}dy\right)^{1/r'}
$$
Thus, we obtain the inequality
$$
\int_{K_y}|a^{ij}(y,\xi_k, \sigma) - a^{ij}_\delta(y,\xi_k, \sigma)|\rho(y)dy \le 2\varepsilon + S\left(\int_{K_y}|a^{ij}(y,\xi_k, \sigma_n) - a^{ij}_\delta(y,\xi_k, \sigma_n)|^{r'}\,dy\right)^{1/r'}
$$
for all $\sigma \in \mathcal P_R(V)$ and for some $n \in\{1,\ldots, M\}$. Let us choose $\delta$ such that the right-hand side of the inequality is less than $3\varepsilon$. The required estimate for the coefficients $b^i$ is obtained in the same way. This completes the proof.
\end{proof}

\section{Proof of Theorem 2.2}
\label{sec5}

Let us first prove Theorem \ref{cor2}.

\begin{proof}[Proof of Theorem \ref{cor2}]
	
	Let us fix a number $R \ge C / \Lambda$ such that the set $\mathcal P_{R}(V)$ is nonempty. The assumptions of the theorem allow us to apply the case of \cite[Theorem 1]{ShSh24} when $\rm{(H_b)}$ is fulfilled. Indeed, by (H1) conditions $\rm{(H_a)}, \rm{(H_b)}$ of \cite[Theorem 1]{ShSh24} are satisfied and (\ref{cor2: cond}) implies the inequality $L_\mu V \le -M$ for all $x$ outside some ball centered at zero, all $\sigma \in \mathcal P(V)$ and some $M > 0$. By \cite[Theorem 1]{ShSh24} for all $\sigma \in \mathcal P_R(V)$ there exists a probability solution $\mu$ to the equation
	$$
	L^*_\sigma \mu =
	\sum_{i,j=1}^d\partial_{x_i}\partial_{x_j}\bigl(a^{ij}(x,\sigma)\mu\bigr)-\sum_{i=1}^d\partial_{x_i}\bigl(b^i(x,\sigma)\mu\bigr)=0.
	$$
	According to \cite[Theorem 2.3.2]{book}, we have the estimate
	$$
	\int_{\mathbb R^d} Vd\mu \le \frac{C}{\Lambda}.
	$$
	Hence $\mu \in \mathcal P_{R}(V)$. 
	
	Let $\Phi(\sigma)$ denote the set of all solutions $\mu\in\mathcal P_{R}(V)$ to the equation $L^*_\sigma\mu = 0$. This is a nonempty and convex subset of the convex compact set $\mathcal P_{R}(V)$. Let us verify that the graph of the mapping $\sigma \mapsto \Phi(\sigma)$ is closed. Suppose that a sequence $\sigma_n\in\mathcal P_{R}(V)$ converges weakly to a measure $\sigma\in\mathcal P_{R}(V)$. Then the sequence $\sigma_n$ converges $V$-weakly to the measure $\sigma$.
	We verify that $L^*_\sigma \mu = 0$. Let $\varphi\in C_0^\infty(\mathbb R^d)$ and the following identity is fulfilled
	$$
	\int_{\mathbb R^d} L_{\sigma_n}\varphi d\mu_n = 0.
	$$
	Let us denote the compact support of $\varphi$ by $K$. We have
	$$
	\int_{\mathbb R^d} L_\sigma \varphi d\mu = \int_K (L_\sigma \varphi - L_{\sigma_n}\varphi) d\mu_n + \int_K L_{\sigma,\delta} \varphi d(\mu - \mu_n) + \int_K \left(L_\sigma - L_{\sigma,\delta}\right) \varphi d(\mu - \mu_n).
	$$
	where $L_{\sigma,\delta}$ is an operator with the coefficients $a^{ij}_\delta, b^i_\delta$. The procedure for constructing $a^{ij}_\delta, b^i_\delta$ is described before Lemma \ref{lemma: approx}. Therefore, the following estimate holds
	\begin{multline*}
		\left|\int_{\mathbb R^d} L_\sigma \varphi d\mu\right| \le \left|\int_K (L_\sigma \varphi - L_{\sigma_n}\varphi) d\mu_n\right| + \left|\int_K L_{\sigma,\delta} \varphi d(\mu - \mu_n)\right| +\\ 
		C(\varphi)\int_K \|A_\sigma-A_{\sigma, \delta}\|+|b_\sigma-b_{\sigma, \delta}| d\mu + C(\varphi)\sup_k\int_K \|A_\sigma-A_{\sigma, \delta}\|+|b_\sigma-b_{\sigma, \delta}| d\mu_k,
	\end{multline*}
	where we denote
	$$
	C(\varphi)=\sup_x|\nabla\varphi(x)|+\sup_x\|D^2\varphi(x)\|.
	$$
	Let the cube $Q=Q_y\times Q_z$ contain the closure of the cube $K$ and the distance from the boundary of the cube $Q_y$ to the boundary of the cube $K_y$ be greater than one. Let $\eta \in C_0^\infty(\mathbb R^{d-m})$, $0 \le \eta \le 1$ and $\eta(z) = 1$ for all $z \in Q_z$. According to Lemma \ref{lemma: regular}, we have $\mu_n \in \mathcal M_{K, S}(\eta)$ for some $S > 0$. Then by Lemma \ref{lemma: converg} we obtain
	\begin{multline*}
		\left|\int_{\mathbb R^d} L_\sigma \varphi d\mu\right| \le C(\varphi)\int_K \|A_\sigma-A_{\sigma, \delta}\|+|b_\sigma-b_{\sigma, \delta}| d\mu + C(\varphi)\sup_k\int_K \|A_\sigma-A_{\sigma, \delta}\|+|b_\sigma-b_{\sigma, \delta}| d\mu_k,
	\end{multline*}
	as $n$ tends to infinity. By Lemma \ref{lemma: approx} we get
	$$
	\int_{\mathbb R^d} L_\sigma \varphi d\mu = 0.
	$$	
	Thus, the graph of the mapping $\sigma \mapsto \Phi(\sigma)$ is closed. According to the Kakutani–Ky Fan theorem (see, for instance, \cite{BS17}), the mapping $\sigma \mapsto \Phi(\sigma)$ has a fixed point $\mu$, which is a solution to equation (\ref{eq1}).
\end{proof}

Let us prove Theorem \ref{main_result}.
\begin{proof}[Proof of Theorem \ref{main_result}]
	
	Let $B_n$ denote the ball with center at zero and radius $n$ in $\mathbb R^d$. Let $\varphi_n\in C_0^\infty(\mathbb R^d)$, $0 \le \varphi_n \le 1$, $\varphi_n = 1$ on $B_n$ and $\varphi_n = 0$ outside $B_{n+1}$. For every probability measure $\mu$ we denote
	$$I_{\mu,n} = \int_{\mathbb R^d}\varphi_nd\mu.$$
	Let $\delta_0$ denote the probability measure with support at zero. We introduce the operator
	$$
	L_{\mu, n} f(x) = \varphi_n(x)L_{\varphi_n\mu + (1-I_{\mu,n})\delta_0}f(x) - \Lambda(1 - \varphi_n(x))V(x),
	$$
	By ${\rm (H1.2)}$ the inequality
	$$
	L_{\mu,n}V(x) \le C_n - \Lambda V(x)
	$$	
	holds for all $x\in \mathbb R^d$ and $\mu \in \mathcal P(V)$, where the number $C_n$ depends only on $V$ and $n$. By~Theorem~\ref{cor1} for all $n$ there exists a probability solution $\mu_n \in \mathcal P(V)$ to the equation $L_{\mu,n}^*\mu = 0$.
	
	Let us show that $\mu_n \in \mathcal P_{R}(V)$, where the number $R$ does not depend on $n$. We denote 
	$$
	\nu_n = \varphi_n\mu_n + (1-I_{\mu_n,n})\delta_0
	$$
	By ${(\rm H1.2)}$ the coefficients of $L_{\mu_n,n}$ are bounded on $B_{n+1}$ and $L_{\mu_n,n}W(x,\nu_n) = -\Lambda V(x)$ for all~$x$~outside the ball $B_{n+1}$. Then the function $L_{\mu_n,n}W(x,\nu_n)$ is integrable with respect to $\mu_n$. Applying \cite[Theorem 2.3.2]{book} to $\Psi(x) = \max\{L_{\mu_n,n}W(x,\nu_n),0\},\,\Phi(x) = \max\{-L_{\mu_n,n}W(x,\nu_n),0\}$, we get
	\begin{equation}
		\label{ineq: nonneg}
		\int_{\mathbb R^d} L_{\mu_n,n} W(x,\nu_n)\,d\mu_n \ge 0.
	\end{equation}
	In addition, we have
		\begin{multline*}
		\int_{\mathbb R^d}L_{\mu_n, n}W(x,\nu_n)d\mu_n =\\
		\int_{\mathbb R^d} L_{\nu_n}W(x, \nu_n)d\nu_n - \Lambda \int_{\mathbb R^d}(1 - \varphi_n(x))V(x)d\mu_n -
		(1 - I_{\mu_n,n})	L_{\nu_n}W(0, \nu_n).
	\end{multline*}
Since the measure $\nu_n$ has a compact support, we can substitute $\nu_n$ in ${(\rm H3.1)}$. We have
$$
\int_{\mathbb R^d}L_{\mu_n, n}W(x,\nu_n)d\mu_n \le C - \Lambda \int_{\mathbb R^d} V(x)d\nu_n - \Lambda \int_{\mathbb R^d}(1 - \varphi_n(x))V(x)d\mu_n -
(1 - I_{\mu_n,n})	L_{\nu_n}W(0, \nu_n).
$$
The measure $(\nu_n - \varphi_n\mu_n)$ is nonnegative. Then we obtain
$$
\int_{\mathbb R^d}L_{\mu_n, n}W(x,\nu_n)d\mu_n \le C - \Lambda \int_{\mathbb R^d} V(x)d\mu_n - (1 - I_{\mu_n,n})	L_{\nu_n}W(0, \nu_n).
$$
By ${(\rm H3.2)}$ we have
$$
-L_{\nu_n}W(0, \nu_n) \le C_1 + C_2H(0) + C_2 \int_{\mathbb R^d}Hd\mu_n
$$
and there exists a number $r > 0$ such that $H(x) \le \varepsilon V(x)$ for all $x$ outside the ball $B_r$ of radius~$r$ centered at zero. Then
$$
\int_{\mathbb R^d}Hd\mu_n \le \max_{x\in B_r} H(x) + \varepsilon\int_{\mathbb R^d\setminus B_r}Vd\mu_n.
$$
Choosing $\varepsilon > 0$ such that $\varepsilon C_2 \le \Lambda / 2$, we obtain
$$
	\int_{\mathbb R^d}L_{\mu_n, n}W(x,\nu_n)d\mu_n \le M - \frac{\Lambda}{2}\int_{\mathbb R^d} V(x)d\mu_n,
$$
where the number $M$ does not depend on $n$. Combining this estimate with inequality (\ref{ineq: nonneg}), we obtain
$$
\int_{\mathbb R^d} V(x)d\mu_n \le \frac{2M}{\Lambda}.
$$
Thus, $\mu_n \in \mathcal P_{R}(V)$ for $R = 2M/\Lambda$ and there exists a subsequence $\mu_{n_k}$ converging $V$-weakly to a measure $\mu \in \mathcal P_{R}(V)$. 

Let us show that $\mu$ is a solution to equation (\ref{eq1}). Suppose that $\psi\in C_0^\infty(\mathbb R^d)$, then the following equality is fulfilled
$$
\int_{\mathbb R^d} L_{\mu_n}\varphi d\mu_n = 0.
$$
We denote the compact support of $\psi$ by $K$. One has the equality
$$
\int_{\mathbb R^d} L_\mu \psi d\mu = \int_K (L_\mu \psi - L_{\mu_n, n}\varphi) d\mu_n + \int_K L_{\mu,\delta} \psi d(\mu - \mu_n) + \int_K \left(L_\mu - L_{\mu,\delta}\right) \psi d(\mu - \mu_n).
$$
Therefore, the following estimate holds
\begin{multline*}
	\left|\int_{\mathbb R^d} L_\mu \psi d\mu\right| \le \left|\int_K (L_\mu \psi - \varphi_nL_{\nu_n}\psi) d\mu_n\right| + \Lambda\left|\int_K (1-\varphi_n) V d\mu_n\right| + \left|\int_K L_{\mu,\delta} \psi d(\mu - \mu_n)\right| +\\ 
	C(\psi)\int_K \|A_\mu-A_{\mu, \delta}\|+|b_\sigma-b_{\mu, \delta}| d\mu + C(\psi)\sup_k\int_K \|A_\mu-A_{\mu, \delta}\|+|b_\mu-b_{\mu, \delta}| d\mu_k,
\end{multline*}
where we denote
$$
C(\psi)=\sup_x|\nabla\psi(x)|+\sup_x\|D^2\psi(x)\|.
$$
Let the cube $Q=Q_y\times Q_z$ contain the closure of the cube $K$ and the distance from the boundary of the cube $Q_y$ to the boundary of the cube $K_y$ be greater than one. Let $\eta \in C_0^\infty(\mathbb R^{d-m})$, $0 \le \eta \le 1$ and $\eta(z) = 1$ for all $z \in Q_z$. According to Remark \ref{rm: reg_with_poten}, we have $\mu_n \in \mathcal M_{Q, S}(\eta)$ for some $S > 0$. Since the sequence $\nu_n$ converges weakly to the measure $\mu$ and $\nu_n \in \mathcal P_{R'}(V)$ for $R' = R + V(0)$, by Lemma \ref{lemma: converg} we obtain
$$
\lim_{n\rightarrow\infty}\left|\int_K (L_\mu \psi - \varphi_nL_{\nu_n}\psi) d\mu_n\right| = 0.
$$
Therefore, we have
\begin{multline*}
	\left|\int_{\mathbb R^d} L_\mu \psi d\mu\right| \le C(\varphi)\int_K \|A_\sigma-A_{\sigma, \delta}\|+|b_\sigma-b_{\sigma, \delta}| d\mu + C(\varphi)\sup_k\int_K \|A_\sigma-A_{\sigma, \delta}\|+|b_\sigma-b_{\sigma, \delta}| d\mu_k,
\end{multline*}
as $n$ tends to infinity. By Lemma \ref{lemma: approx} we get
$$
\int_{\mathbb R^d} L_\mu \psi d\mu = 0,
$$	
Thus, $\mu$ is a solution to equation (\ref{eq1}). This completes the proof.
\end{proof}

\vspace*{0.2cm}

\begin{proof}[Proof of Corollary \ref{cor1}]
	By inequality (\ref{cor: cond2}) assumption ${\rm (H3.1)}$ is fulfilled for $W(x,\mu)=V(x).$ Then we have
	\begin{multline*}
		-L_\mu W(0, \mu) = -\sum_{i,j = 1}^d a^{ij}(0,\mu) \partial_{x_i}\partial_{x_i}V(0) - \sum_{i,j = 1}^d b^i(0,\mu) \partial_{x_i}V(0) \le\\ 
	d^2C(V) \max_{i,j}\left(|a^{ij}(0,\mu)| +|b^i(0,\mu)|\right)
	\end{multline*}
	where $C(V) = |\nabla V(0)|+\|D^2 V(0)\|.$ By assumption (\ref{cor: cond1}) we obtain the estimate
	$$
	-L_\mu W(0, \mu) \le d^2C(V)\left(C_1 + C_2 \int_{\mathbb R^d}H(x)\mu(dx)\right).
	$$
	Thus, the conditions of Theorem \ref{main_result} is satisfied and equation {\rm (\ref{eq1})} has a probability solution in the set $\mathcal P(V)$. This completes the proof.
\end{proof}

\vspace*{0.2cm}

The proof of Theorem \ref{th2} is a simplified version of the proof of Theorem \ref{main_result}.

\begin{proof}[Proof of Theorem \ref{th2}]
	Let $B_n$ denote the ball with center at zero of radius $n$ in $\mathbb R^d$. Suppose that $\varphi_n\in C_0^\infty(\mathbb R^d)$, $0 \le \varphi_n \le 1$, $\varphi_n = 1$ on $B_n$ and $\varphi_n = 0$ outside $B_{n+1}$. Let us introduce the operator 
	$$ 
	L_{\mu, n} f(x) = \varphi_n(x)L_{\varphi_n\mu}f(x) - \Lambda(1 - \varphi_n(x))V(x). 
	$$
	The inequality
	$$ 
	L_{\mu,n}V \le C_n - \Lambda V. 
	$$
	holds for all $x\in \mathbb R^d$ and $\mu \in \mathcal P(V)$, where the number $C_n$ depends only on $W$ and $n$. By~Theorem \ref{cor1} for all $n$ there exists a probability solution $\mu_n \in \mathcal P(V)$ to the equation $L_{\mu,n}^*\mu = 0$. 
	
	Let us show that $\mu_n \in \mathcal P_{R}(V)$, where the number $R$ does not depend on $n$. By ${(\rm H1.2)}$ the coefficients of $L_{\mu_n,n}$ are bounded on $B_{n+1}$ and $L_{\mu_n,n}W(x,\varphi_n\mu_n) = -\Lambda V(x)$ for all $x$ outside the ball $B_{n+1}$. Then the function $L_{\mu_n,n}W(x,\varphi_n\mu_n)$ is integrable with respect to the measure $\mu_n$. Applying \cite[Theorem 2.3.2]{book} to $\Psi(x) = \max\{L_{\mu_n,n}W(x,\varphi_n\mu_n),0\},\,\Phi(x) = \max\{-L_{\mu_n,n}W(x,\varphi_n\mu_n),0\}$, we get
	\begin{equation}
		\label{ineq: nonneg2}
		\int_{\mathbb R^d} L_{\mu_n,n} W(x,\varphi_n\mu_n)\,d\mu_n \ge 0.
	\end{equation}
Since the measure $\varphi_n\mu_n$ has a compact support, we can substitute $\varphi_n\mu_n$ in ${(\rm H3.1)}$. Then we have
$$
		\int_{\mathbb R^d}L_{\mu_n, n}W(x,\varphi_n\mu_n)d\mu_n \le C - \Lambda \int_{\mathbb R^d}V(x)d\mu_n.
$$
Combining this estimate with inequality (\ref{ineq: nonneg2}), we obtain
$$
\int_{\mathbb R^d} V(x)d\mu_n \le \frac{C}{\Lambda}.
$$
Thus, $\mu_n \in \mathcal P_{R}(V)$ for $R = C/\Lambda$ and there exists a subsequence $\mu_{n_k}$ converging $V$-weakly to a measure $\mu \in \mathcal P_{R}(V)$. The proof that $\mu$ is a solution to equation (\ref{eq1}) is identical to the proof of this fact in Theorem \ref{main_result}. We just need to replace $\nu_n$ by $\varphi_n \mu_n$ in that proof.
\end{proof}

\vspace*{0.2cm}

\centerline{\bf Acknowledgements.}

\vspace*{0.2cm}

The authors are grateful to S.V. Shaposhnikov for fruitful discussions and valuable remarks.

This paper is supported by the Theoretical Physics and Mathematics Advancement Foundation~``BASIS".


\begin{thebibliography}{}
	
\bibitem{ABR19}
H. AlRachid, M. Bossy, C. Ricci, L. Szpruch,
New particle representations for ergodic McKean--Vlasov SDEs.
ESAIM: Proceedings and Surveys 65 (2019): 68-83.
https://doi.org/10.1051/proc/201965068

\bibitem{A20}
Yu. Averboukh, 
Viability analysis of the first-order mean field games, 
ESAIM: Control, Optimisation and Calculus of Variations 26 (2020): 33.
https://doi.org/10.1051/cocv/2019013.

\bibitem{AV26}
Yu. Averboukh, A. Volkov,
Necessary condition of asymptotic stability for non-local continuity equation, 
Journal of Mathematical Sciences (2026): 1-7.
https://doi.org/10.1007/s10958-025-08170-9

\bibitem{book}
V.I. Bogachev, N.V. Krylov,  M. R\"ockner, S.V. Shaposhnikov
Fokker--Planck--Kolmogorov equations,
American Mathematical Society, 2015.

\bibitem{BRSsup23}
V.I. Bogachev, M. R\"ockner, S.V. Shaposhnikov, 
On the Ambrosio--Figalli--Trevisan superposition principle for probability solutions to
Fokker--Planck--Kolmogorov equations,
Journal of Dynamics and Differential Equations 33(2) (2021): 715-739.
https://doi.org/10.1007/s10884-020-09828-5

\bibitem{ZV23}
V.I. Bogachev, M. R\"ockner, S.V. Shaposhnikov,
Zvonkin’s transform and the regularity of solutions to double divergence form elliptic equations.
Communications in Partial Differential Equations 48(1), 119–149. https://doi.org/10.1080/03605302.2022.2139724

\bibitem{BS24}
V.I. Bogachev, S.V. Shaposhnikov, 
Nonlinear Fokker-Planck-Kolmogorov equations,
Uspekhi Matematicheskikh Nauk 79(5) (2024): 3-60.
https://doi.org/10.4213/rm10202e

\bibitem{BSS25}
V.I. Bogachev, S. V. Shaposhnikov, D. V. Shatilovich,
Doubling variables and uniqueness of probability solutions to degenerate stationary Kolmogorov equations,
Journal of Dynamics and Differential Equations (2026): 1-19.
https://doi.org/10.1007/s10884-026-10497-z

\bibitem{BS17}
V.I. Bogachev, O.G. Smolyanov, 
Topological vector spaces and their applications, 
Cham: Springer, 2017.
https://doi.org/10.1007/978-3-319-57117-1

\bibitem{ButS17}
O. Butkovsky, M. Scheutzow, 
Invariant measures for stochastic functional differential equations,
Electronic Journal of Probability 22 (2017) 1-23. 
https://doi.org/10.1214/17-EJP122

\bibitem{CG15}
P. Cardaliaguet, P.J. Graber, 
Mean field games systems of first order, ESAIM: Control, Optimisation and Calculus of Variations 21(3) (2015): 690-722.
https://doi.org/10.1051/cocv/2014044

\bibitem{CD18}
R. Carmona, F.Delarue, 
Probabilistic theory of mean field games with applications I-II, 
Berlin: Springer Nature, 2018.
https://doi.org/10.1007/978-3-319-58920-6,
https://doi.org/10.1007/978-3-319-56436-4

\bibitem{J16}
A. J\"ungel, 
Entropy methods for diffusive partial differential equations, 
Cham: Springer, 2016.
https://doi.org/10.1007/978-3-319-34219-1

\bibitem{HSS21}
W. R. Hammersley, D. Siska, L. Szpruch, 
McKean–Vlasov SDEs under measure dependent Lyapunov conditions,
Annales de l'Institut Henri Poincar\'e, Probabilités et Statistiques 57(2) (2021): 1032-1057.

\bibitem{HRW21}
X. Huang, P. Ren, F. Y. Wang, 
Distribution dependent stochastic differential equations, Frontiers of Mathematics in China, 16(2), (2021): 257-301. 
https://doi.org/10.1007/s11464-021-0920-y

\bibitem{IW89}
N. Ikeda, S. Watanabe, 
Stochastic differential equations and diffusion processes,
North-Holland Mathematical Library, 1981.

\bibitem{Kac56}
M. Kac, 
Foundations of kinetic theory, 
Proceedings of the Third Berkeley Symposium on Mathematical Statistics and Probability, University of California Press (1956), 171–197.

\bibitem{K19}
V.N. Kolokoltsov, 
Differential equations on measures and functional spaces, 
Switzerland: Springer International Publishing, 2019.
https://doi.org/10.1007/978-3-030-03377-4

\bibitem{KT19}
V.N. Kolokoltsov, M. Troeva,
On mean field games with common noise and McKean-Vlasov SPDEs, 
Stochastic Analysis and Applications 37(4) (2019): 522-549.
https://doi.org/10.1080/07362994.2019.1592690

\bibitem{LL07}
J. M. Lasry, P. L. Lions, 
Mean field games, 
Japanese journal of mathematics 2(1), 229-260.
https://doi.org/10.1007/s11537-007-0657-8

\bibitem{LiSt}
H. Lee, G. Trutnau, 
Existence and regularity of infinitesimally invariant measures, transition functions and time-homogeneous Ito-SDEs,
Journal of Evolution Equations 21(1) (2021): 601-623.
https://doi.org/10.1007/s00028-020-00593-y

\bibitem{McK66}
H.P. McKean, 
A class of markov processes associated with nonlinear parabolic equations, 
Proceedings of the National Academy of Sciences of the United States of America 56(6) (1966): 1907–1911.

\bibitem{McK67}
H.P. McKean, 
Propagation of chaos for a class of non-linear parabolic equations, 
Stochastic Differential Equations (1967): 41–57.

\bibitem{MV20}
Y.S. Mishura, A.Y. Veretennikov, 
Existence and uniqueness theorems for solutions of McKean–Vlasov stochastic equations, Theory of Probability and Mathematical Statistics, 103 (2020): 59–101.

\bibitem{R23}
A. Rebucci, 
Regularity results and new perspectives for degenerate Kolmogorov equations,
Ph.D. thesis, Universit\`a di Parma, 2023.

\bibitem{SS25}
S.V. Shaposhnikov, D.V. Shatilovich,
Analysis of mean field games via Fokker--Planck--Kolmogorov equations: existence of equilibria, Nonlinear Analysis 270 (2026).

\bibitem{ShSh24}
S.V. Shaposhnikov, D.V. Shatilovich,
Khasminskii's Theorem for the Kolmogorov Equation with Partially Degenerate Diffusion Matrix, Mathematical Notes 115(3) (2024): 427--438.
https://doi.org/10.1134/S0001434624030155.

\bibitem{Vlasov68}
A.A. Vlasov, 
The vibrational properties of an electron gas, 
Soviet Physics Uspekhi 10(6) (1968): 721-733.
https://doi.org/10.1070/PU1968v010n06ABEH003709

\bibitem{W21}
F.Y. Wang, 
Exponential ergodicity for non-dissipative McKean-Vlasov SDEs, 
Bernoulli 29(2) (2023): 1035-1062.
https://doi.org/10.3150/22-BEJ1489

\bibitem{Z21}
S.-Q. Zhang, 
Existence and non-uniqueness of stationary distributions for distribution dependent SDEs,
Electronic Journal of Probability 28 (2023): 1-34.
https://doi.org/10.1214/23-EJP981

\end{thebibliography}
\end{document}